\documentclass[12pt,letterpaper]{article}

\usepackage[a4paper,
  left = 25mm,right=25mm,
  top  = 22mm,bottom=25mm]{geometry}

\usepackage{amsthm, amssymb}
\usepackage{mathtools}
\usepackage{faktor}
\usepackage{psfrag}

\usepackage{tikz}
\usetikzlibrary{positioning}
\usetikzlibrary{calc}
\usetikzlibrary{matrix}
\usetikzlibrary{decorations.pathreplacing}

\usepackage{authblk}
\usepackage{pifont}

\usepackage{booktabs}

\def\plus{{\boldsymbol{\dag}}}

\newcommand{\R}{\mathbb R}

\newcommand\norm[1]{\Vert#1\Vert}

\newcommand\set[1]{{\{#1\}}}

\newcommand{\edot}{\,\cdot\,}

\DeclareMathOperator{\ran}{ran}
\DeclareMathOperator{\id}{Id}

\DeclareMathOperator*{\argmin}{arg\,min}

\newcommand{\eps}{\epsilon}

\newtheorem{theorem}{Theorem}
\newtheorem{proposition}[theorem]{Proposition}
\newtheorem{example}[theorem]{Example}
\newtheorem{remark}[theorem]{Remark}

\theoremstyle{definition}

\newtheorem{definition}[theorem]{Definition}

\usepackage{makecell}

\usepackage{enumitem}
\usepackage{xcolor}
\usepackage{amssymb}

  \newenvironment{tritemize}
  {\begin{itemize}[label=\textcolor{green!30!gray}{\raisebox{0.2ex}{$\blacktriangleright$}}]}
  {\end{itemize}}
  
\usepackage[colorlinks=true, allcolors=blue]{hyperref}

\newcommand{\N}{\mathcal{N}}

\title{Data-Consistent Learning of Inverse Problems}

\author[1]{Markus Haltmeier}
\author[2]{Gyeongha Hwang}

\affil[1]{Department of Mathematics, University of Innsbruck, Austria.}

\affil[2]{Department of Mathematics, Yeungnam University, Gyeongsan, Korea.}

\date{} 

\begin{document}
\maketitle

\begin{abstract}
Inverse problems are inherently ill-posed, suffering from non-uniqueness and instability. Classical regularization methods provide mathematically well-founded solutions, ensuring stability and convergence, but often at the cost of reduced flexibility or visual quality. Learned reconstruction methods, such as convolutional neural networks, can produce visually compelling results, yet they typically lack rigorous theoretical guarantees. DC (DC) networks address this gap by enforcing the measurement model within the network architecture. In particular, null-space networks combined with a classical regularization method as an initial reconstruction define a convergent regularization method. This approach preserves the theoretical reliability of classical schemes while leveraging the expressive power of data-driven learning, yielding reconstructions that are both accurate and visually appealing.

\end{abstract}

\section{Introduction}
\label{sec:DC-intro}

We consider the problem of solving linear inverse problems of the form
\begin{equation} \label{eq:ip}
    y = A x_0 + z,
\end{equation}
where \(A \colon X \to Y\) is a known linear operator between Banach spaces \(X\) and \(Y\). Here, \(x_0 \in X\) denotes the unknown signal, \(y \in Y\) the measured data, and \(z \in Y\) represents additive measurement noise with bound $\| z \| \leq \delta$. The objective is to recover a stable and accurate estimate of \(x\) from the noisy observations \(y\). Since \(A\) may be ill-conditioned or non-invertible, the problem is generally ill-posed: direct inversion is either unstable or non-unique. To address this, it is essential to incorporate additional prior information \cite{morozov2012methods, scherzer2009variational,tikhonov1977solutions}.

\paragraph*{Regularization methods:}

Classical approaches to inverse problems rely on regularization, which incorporates prior information by penalizing undesirable properties of the solution. Common examples include Tikhonov regularization, spectral filtering, and iterative schemes \cite{engl1996regularization,ito2014inverse, scherzer2009variational,schuster2012regularization}. For instance, Tikhonov regularization replaces unstable inversion with a well-posed minimization problem of the form 
\begin{equation} \label{eq:tik}
\min_{x \in X} \left\{ \|A x - y\|^p + \alpha \|x\|^q \right\},
\end{equation}
where \(p, q > 1\) are constants, and \(\alpha > 0\) is the regularization parameter that balances data fidelity and stability. Such methods enforce DC up to the noise level while suppressing noise amplification and model instability.

A key property of these classical methods is that, under a suitable parameter choice $\alpha = \alpha(\delta)$, they are both DC and convergent: as the noise level $\delta \to 0$,  for any solution $x_\alpha$ of \eqref{eq:tik}, the residual $\|A x_\alpha - y\|$ remains within the noise bound, while the reconstruction $x_\alpha$ approaches the true signal $x_0$. However, these methods often rely on handcrafted regularization functionals or filter functions, which may fail to capture the complex structures present in high-dimensional signals, such as images.

\paragraph*{Learned reconstruction:}

In recent years, data-driven neural networks have emerged as powerful alternatives to classical regularization. These models learn effective priors directly from training data, often producing visually superior reconstructions compared to classical schemes \cite{adler2017solving,lucas2018using,arridge2019solving,mccann2017convolutional,wang2020deep,jin2017deep,li2018nett,han2016deep,gilton2021deep,kirisits2025regularization}. 

Roughly, these methods can be grouped into explicit and implicit approaches. In an explicit approach \cite{han2016deep,jin2017deep}, a known network \((\mathcal{B}_\theta)_{\theta \in \Theta}\) with \(\mathcal{B}_\theta \colon Y \to X\) is trained to map data \(y = A x_0 + z\) to the corresponding signal \(x_0\) for paired training data \((x_0, y)\). In an implicit approach \cite{li2018nett},  the reconstruction \(\hat{x}\) is defined implicitly via the network, for example through a fixed-point equation arising from an optimization problem. Specifically, in the context of learned regularization functionals, one minimizes  \(\|A x - y\|^2 + \alpha \, \Omega_\theta(x)\), where \(\Omega_\theta\) is a regularizer trained on data. Once the regularizer is learned, the framework largely aligns with classical variational regularization theory \cite{scherzer2009variational}. A broad overview of learned regularization functionals and training strategies can be found in \cite{hertrich2025learning}.

In this chapter, we focus on the explicit learning approach for solving inverse problems. In this context, three central challenges remain:
\begin{enumerate}[label=(\alph*)]
    \item \emph{Data-consistency (DC):} \label{E:dc}
    How to ensure consistency with measurements.
    \item \emph{Convergence:} \label{E:convergence}
    When do learned reconstructions yield a convergent regularization.
    \item \emph{Convergence rates:} \label{E:rates} 
    How to obtain quantitative estimates between the ground truth signal and learned reconstructions.
\end{enumerate}

To address these questions, we restrict our attention to two-step methods \(g_{\theta, \delta} = f_\theta  \circ \mathcal{B}_\delta\), where \(f_\theta \colon X \to X\) is an image-space network and \(\mathcal{B}_\delta \colon Y \to X\) is an initial reconstruction map. While deep networks can model complex image and noise statistics, they frequently lack explicit mechanisms to guarantee near DC,
\begin{equation} \label{eq:DC}
    \|A (g_{\theta, \delta}(y)) - y\| \leq \tau \cdot \delta \qquad \text{as } \delta \to 0 .
\end{equation}
As a result, they may introduce hallucinated or DC features, particularly in underdetermined settings \cite{bhadra2021hallucinations}.   

Informally, we call \emph{hallucination} the addition of reconstruction features $h \neq 0$ such that either
(i) they are not supported by the measurements, with $A(x+h)$ outside the measurement-uncertainty set, or
(ii) they are consistent with the data, with $A(x+h)$ within the measurement-uncertainty set, yet not warranted by the measurements alone.
Consistent hallucinations are, to some extent, unavoidable: adjusting DC components is inherent to any reconstruction method.
This adjustment is guided by signal and noise statistics, and what counts as proper DC artifact removal for one image class may be perceived as hallucination for another.
In contrast, hallucinations not supported by the measurements are unacceptable, as they contradict the only information provided by the data.

As we argue below, near data consistency of $g_{\theta, \delta}$ is key to establishing both convergence and convergence rates. We realize this via the null-space approach: the deterministic component $\mathcal{B}_\delta$ is near DC, and the learned component $f_\theta$ acts only on $\ker(A)$.

 \paragraph*{Scope of this chapter:}

DC refers to the requirement that a reconstruction \(g_{\theta, \delta}(y)\) reproduces the measured data through the forward operator \(A\) within the noise bound. While classical regularization methods naturally yield approximate DC, most learning-based architectures do not, which limits their theoretical interpretability. 

This chapter addresses the  gap between classical and learning-based reconstruction methods by formulating DC neural architectures within the framework of regularization theory. We focus in particular on the concept of \emph{null-space networks} introduced in \cite{schwab2019deep}, which modify initial reconstructions only within the null space of \(A\), thereby preserving DC by design. Building on this idea, we discuss how deep learning models can be interpreted as convergent regularization methods and thus retains the reliability of classical approaches while achieving superior empirical performance. The focus of the chapter is on theoretical aspects, including convergence and convergence rates. The key message is that DC learning offers the same theoretical guarantees as classical handcrafted regularization methods for solving inverse problems.

\section{Regularization background}
\label{sec:DC-inverse}

Let $A$ be a bounded linear operator between two Banach spaces $X$, $Y$.

\subsection{Ill-posedness}

The defining characteristic of inverse problems \eqref{eq:ip} is their \emph{ill-posedness}. Even in the case of noise-free data, the equation \(y = A x\) may lack a unique or existing solution and, in the presence of noise, additionally suffers from instability. Formally, an inverse problem violates at least one of the following properties:
\begin{enumerate}[label=(I\arabic*)]
    \item \label{ip1} \emph{Non-uniqueness:} There exist \(x_1 \neq x_2 \in X\) such that \(A x_1 = A x_2\).
    \item \label{ip2} \emph{Non-existence:} For some \(y \in Y\), there exists no \(x \in X\) satisfying \(A x = y\).
    \item \label{ip3} \emph{Instability:} Small differences in data \(\|A x_1 - A x_2\|\) do not imply small differences in the corresponding solutions \(\|x_1 - x_2\|\).
\end{enumerate}

If any of these conditions hold, the operator \(A\) fails to possess a continuous inverse, and direct inversion of \eqref{eq:ip} becomes infeasible. To obtain stable reconstructions, one therefore resorts to \emph{regularization methods}.

Regularization typically proceeds in two conceptual steps. First, to address non-uniqueness and non-existence, one restricts the domain and range of the forward operator to subsets \(\mathcal{D} \subseteq X\) and \(\mathcal{R} \subseteq Y\) such that the restricted operator \(A_{\mathrm{res}} \colon \mathcal{D} \to \mathcal{R}\) becomes bijective. For any exact data \(y \in \mathcal{R}\), the inverse problem then has a unique solution in \(\mathcal{D}\), given by \(x = A_{\mathrm{res}}^{-1}(y)\).

Second, to mitigate instability \ref{ip3}, one introduces a family of continuous, stable mappings \(B_\alpha \colon Y \to X\), depending on a regularization parameter \(\alpha > 0\), which approximate \(A_{\mathrm{res}}^{-1}\) in a suitable sense as \(\alpha \to 0\); see Section~\ref{sec:reg} for a precise definition.

The choice of the admissible set \(\mathcal{D}\) is crucial, as it defines the class of desired reconstructions and acts as a selection principle for choosing a particular solution of the noise-free inverse problem. However, \(\mathcal{D}\) is rarely known exactly and is often only implicitly defined. For instance, in X-ray computed tomography (CT), \(\mathcal{D}\) would represent all physically plausible attenuation distributions within a patient; an extremely complex and poorly characterized function class.

Among the most successful regularization frameworks are variational methods and their extensions \cite{scherzer2009variational}. In this setting, \(\mathcal{D}\) is defined implicitly via a regularization functional that quantifies deviation from desirable properties. Prominent examples include quadratic Hilbert-space penalties \cite{engl1996regularization}, total variation (TV) regularization \cite{acar1994analysis}, and sparsity-promoting \(\ell^q\)-penalties \cite{daubechies2004iterative,grasmair2008sparse}. These handcrafted regularization functionals are well understood, supported by efficient algorithms, and come with rigorous convergence guarantees. However, their limited expressiveness makes them insufficient to capture the complex statistical and structural characteristics of modern high-dimensional applications, such as those arising in medical imaging.

\subsection{Right inverses}

As discussed in the introduction, the core difficulty of inverse problems stems from their ill-posedness.  
To address potential non-existence and non-uniqueness of solutions, one typically restricts the domain of admissible data and considers right inverses defined on a subset $\mathcal{R} \subseteq Y$.

\paragraph*{Algebraic right inverses:}

We first consider right inverses that are not necessarily linear and start with their purely algebraic form.

\begin{definition}[Right inverse] \label{def:right}
A (possibly non-linear) mapping $B \colon \mathcal{R} \subseteq Y \to X$ is called a \emph{right inverse} of $A$ if
\[
A(B(y)) = y \qquad \text{for all } y \in \mathcal{R}.
\]
\end{definition}

With $\mathcal{D} \coloneqq B(\mathcal{R})$, the operator $A$ restricts to a bijection
\[
A_{\mathrm{res}} \colon \mathcal{D} \to \mathcal{R} \colon \; x \mapsto A x,
\]
and $B_{\mathrm{res}} \colon \mathcal{R} \to \mathcal{D}$ is precisely the inverse of this restricted map.  
A right inverse exists if and only if $\mathcal{R} \subseteq \operatorname{ran}(A)$, and this existence is purely algebraic: for every $y \in \mathcal{R}$ one may choose any $x \in X$ satisfying $A x = y$.  
In contrast, continuity of a right inverse is a far stronger requirement and typically fails for ill-posed operators.

\paragraph*{Continuous right inverses:}

The following result characterises when a continuous right inverse can exist. For linear right inverses, this is classical; see \cite{nashed1987inner}.

\begin{proposition}[Continuous right inverses] \label{prop:right-inverse-cont}
Let $B \colon \operatorname{ran}(A) \to X$ be a continuous right inverse.  
Then $\operatorname{ran}(A)$ is closed.
\end{proposition}

\begin{proof}
Let $\Pi \colon X \to X / \ker(A)$ be the quotient map and let $A_\diamond \colon X / \ker(A) \to \operatorname{ran}(A)$ be the induced bounded linear bijection defined by $A_\diamond([x]) = A x$. Define $S \colon \operatorname{ran}(A) \to X / \ker(A)$ by $S := \Pi \circ B$. Then $S$ is continuous and $A_\diamond \circ S \;=\; A_\diamond(\Pi(B(y))) \;=\; A(B(y)) \;=\; y$ for all $y \in \operatorname{ran}(A)$. 
Since $A_\diamond$ is bijective, this shows $S = A_\diamond^{-1}$, hence $A_\diamond^{-1}$ is continuous. Because $X/\ker(A)$ is a Banach space, $\operatorname{ran}(A)$ is (via $A_\diamond$) isomorphic to a Banach space and therefore complete. A complete subspace of $Y$ must be closed, so $\operatorname{ran}(A)$ is closed.
\end{proof}

Thus, no continuous right inverse can exist when $\operatorname{ran}(A)$ is non-closed.  
This is the generic situation in inverse problems: many integral operators, the Radon transform and all compact operators, have non-closed ranges \cite{engl1996regularization,scherzer2009variational,natterer2001mathematics}.

\paragraph*{Linear right inverses:}

Another natural question concerns the existence of a \emph{linear} right inverse.  
Recall that a mapping $P \colon X \to X$ is called a projection if $P^2 = P$.  
For a bounded linear projection, both $\operatorname{ran}(P)$ and $\ker(P)$ are closed, and 
\[
X = \operatorname{ran}(P) \oplus \ker(P).
\]

\begin{definition}[Complemented subspace]
A closed subspace $V \subseteq X$ is called \emph{complemented} if there exists a bounded linear projection $P$ with $\operatorname{ran}(P) = V$.
\end{definition}

Equivalently, $V \subseteq X$ is complemented if and only if there exists another closed subspace $U$ such that $X = U \oplus V$.  
In Hilbert spaces, every closed subspace is complemented via orthogonal projections,
\begin{align*}
X &= V^\perp \oplus V, \\
V^\perp &\coloneqq \{ u \in X \mid \langle u, v \rangle = 0 \text{ for all } v \in V \}, \\
P(x) &= \arg\min \{ \|x-z\| \mid z \in V \}.
\end{align*}
In contrast, non-complemented subspaces exist in general Banach spaces \cite{lindenstrauss1971complemented}.

\begin{proposition}[Linear right inverses] \label{prop:right-inverse-lin}
\leavevmode
\begin{enumerate}[label=(\alph*)]
\item \label{prop:rightlin1}
$A$ has a linear right inverse on $\operatorname{ran}(A)$ if and only if $\ker(A)$ is complemented.
\item \label{prop:rightlin2}
A linear right inverse on $\operatorname{ran}(A)$ is continuous if and only if $\operatorname{ran}(A)$ is closed.
\end{enumerate}
\end{proposition}

\begin{proof}
\noindent\ref{prop:rightlin1}  
If $A$ admits a linear right inverse $B$ on $\operatorname{ran}(A)$, then $BA$ is a bounded projection on $X$ since $(BA)^2 = B(AB)A = BA$.  
This yields the decomposition $X = \operatorname{ran}(BA) \oplus \ker(BA)$, where $\ker(BA) = \ker(A)$, showing that $\ker(A)$ is complemented.  
Conversely, if $\ker(A)$ is complemented, write $X = X_1 \oplus \ker(A)$.  
Then $A_{\mathrm{res}} \colon X_1 \to \operatorname{ran}(A)$ is bijective, and its inverse $A_{\mathrm{res}}^{-1}$ defines a linear right inverse of $A$.

\smallskip
\noindent\ref{prop:rightlin2}  
If $B$ is continuous, Proposition~\ref{prop:right-inverse-cont} implies that $\operatorname{ran}(A)$ is closed.  
Conversely, if $\operatorname{ran}(A)$ is closed, the restricted mapping $A_{\mathrm{res}} \colon X_1 \to \operatorname{ran}(A)$ is bijective and, by the bounded inverse theorem, bounded.  
Thus its inverse $B = A_{\mathrm{res}}^{-1}$ is continuous.
\end{proof}

In Hilbert spaces, every closed subspace (including $\ker(A)$) is complemented, implying that any bounded linear operator $A$ possesses a linear right inverse.

\begin{proposition}[Right inverses in Hilbert spaces] \label{prop:right-hil}
Let $X$ be a Hilbert space and $P_{\ker(A)} \colon X \to X$ the orthogonal projection onto $\ker(A)$.
\begin{enumerate}[label=(\alph*)]
\item \label{prop:right-h1}
$A$ admits a unique linear right inverse $A^+ \colon \operatorname{ran}(A) \to X$ satisfying $AA^+ = \mathrm{id}_{\operatorname{ran}(A)}$ and $A^+A = \id_X - P_{\ker(A)}$.
\item \label{prop:right-h2}
For every $y \in \operatorname{ran}(A)$,
\(
A^+(y) = \arg\min_{x \in X} \{ \|x\| \mid Ax = y \}
\).
\item \label{prop:right-h3}
$A^+$ is continuous if and only if $\operatorname{ran}(A)$ is closed.
\item \label{prop:right-h4}
If $\operatorname{ran}(A)$ is non-closed, then every right inverse of $A$ is discontinuous.
\end{enumerate}
\end{proposition}

\begin{proof}
In Hilbert spaces, the orthogonal complement $\ker(A)^\perp$ provides a complement of $\ker(A)$, and Propositions~\ref{prop:right-inverse-cont} and \ref{prop:right-inverse-lin} yield items \ref{prop:right-h1}, \ref{prop:right-h3}, and \ref{prop:right-h4}.  
For \ref{prop:right-h2}, any solution of $Ax = y$ decomposes as $x = x_1 + x_2$ with $x_1 \in \ker(A)^\perp$ and $x_2 \in \ker(A)$.  
By Pythagoras’ theorem, $\|x\|^2 = \|x_1\|^2 + \|x_2\|^2$, so the minimal-norm solution is $x_2 = 0$.
\end{proof}

If $X$ and $Y$ are both Hilbert spaces, the mapping $A^+$ extends uniquely to $\operatorname{ran}(A) \oplus \operatorname{ran}(A)^\perp$ by $A^+(y_1 + y_2) = A^+(y_1)$.  
This operator is the \emph{Moore–Penrose inverse} of $A$.  
For further background on generalized inverses in Hilbert and Banach spaces, see \cite{nashed1987inner}.


\subsection{Regularization methods}
\label{sec:reg}

Let $B \colon \ran(A) \subseteq Y \to X$ be a right inverse of $A$, and define $\mathcal{D} \coloneqq \ran(B)$. 
Suppose that $\mathcal{D}_0 \subseteq \mathcal{D}$ denotes a potentially further restricted   class of desirable signals. 
Common notion of a regularization method for the inverse problem \eqref{eq:ip} is of a family of operators $(B_\alpha)_{\alpha > 0}$ together with a parameter choice rule $\alpha^* = \alpha^*(\delta, y^\delta)$; see \cite{engl1996regularization,scherzer2009variational}. For our purpose  the resulting reconstruction $B_{\alpha^*(\delta, y^\delta)}(y^\delta)$ is of particular interest and we combine both ingredients into a single mapping $\mathcal{B}(\delta, y^\delta)$.

\begin{definition}[Regularization method]\label{def:reg}
A function $\mathcal{B} \colon (0, \infty) \times Y \to X$ satisfying
\[
\forall x \in \mathcal{D}_0 \colon  
\lim_{\delta \to 0} \sup \bigl\{ \norm{x - \mathcal{B}(\delta, y^\delta)} \ \big|\ y^\delta \in Y,\, \norm{A(x)- y^\delta} \le \delta \bigr\} = 0
\]
is called a (convergent) regularization method for \eqref{eq:ip} on the signal class $\mathcal{D}_0$. We also write $(\mathcal{B}_\delta)_{\delta > 0}$ instead of $\mathcal{B}$.
\end{definition}

A practical guideline for constructing regularization methods is to approximate the right inverse $B$ pointwise by a family of continuous mappings.

\begin{proposition}[Pointwise approximations as regularizations]\label{prop:point}
Let $(B_\alpha)_{\alpha > 0}$ be a family of continuous operators $B_\alpha \colon Y \to X$ that converge uniformly to $B$ on $A(\mathcal{D}_0)$ as $\alpha \to 0$. Then there exists a function $\alpha_0 \colon (0, \infty) \to (0, \infty)$ such that
\(
\mathcal{B}(\delta, y^\delta) \coloneqq B_{\alpha_0(\delta)}(y^\delta)
\)
defines a regularization method for \eqref{eq:ip} on $\mathcal{D}_0$.
\end{proposition}

\begin{proof}
For any $\varepsilon > 0$, choose $\alpha(\varepsilon)$ such that $\norm{B_{\alpha(\varepsilon)}(y) - x} \le \varepsilon/2$ for all $x \in \mathcal{D}_0$ and  $y = Ax$. Moreover, choose $\tau(\varepsilon)$ so that, for all $z \in Y$ with $\norm{y - z} \le \tau(\varepsilon)$, we have $\norm{B_{\alpha(\varepsilon)}(y) - B_{\alpha(\varepsilon)}(z)} \le \varepsilon/2$. Assume $\tau$ is strictly increasing, continuous, and satisfies $\tau(0+) = 0$, and define $\alpha_0 = \alpha  \tau^{-1}$. Then, for $\norm{y - y^\delta} \le \delta$,
\[
\norm{B_{\alpha_0(\delta)}(y^\delta) - x} 
\le \norm{B_{\alpha_0(\delta)}(y) - x} 
+ \norm{B_{\alpha_0(\delta)}(y) - B_{\alpha_0(\delta)}(y^\delta)} 
\le \tau^{-1}(\delta),
\]
and since $\tau^{-1}(\delta) \to 0$ as $\delta \to 0$, the claim follows.
\end{proof}

\begin{remark}[DC family] \label{rem:DC}
We call $(\mathcal{B}_\delta)_{\delta>0}$ a data-consistent (DC) family if $A \mathcal{B}_\delta A x \to A x$ as $\delta \to 0$ for all $x \in X$. If $B$ is a right inverse and $(\mathcal{B}_\delta)_{\delta>0}$ is a regularization method with $\mathcal{B}_\delta \to B$ pointwise on $\operatorname{ran}(A)$, then $\|A \mathcal{B}_\delta A x - A x\| \le \|A\|\,\|\mathcal{B}_\delta A x - B A x\| \to 0$ as $\delta \to 0$. Thus any consistent regularization method is a DC family. However, this requires that the limit $B$ is a right inverse, hence strictly DC and unbiased. In practice, methods may fail to be DC when the limiting $B$ is not a right inverse. An example of this issue is limited-data tomography with $A = M R$ (full Radon transform $R$, mask $M$). Commonly, filtered backprojection $B$ (which, in the noise-free limit with full data, is a right inverse of $R$) is applied to masked data $y = A x = M R x$, yielding $A B A  \neq A$, so DC fails.
\end{remark}

\paragraph*{Tikhonov regularization:}

A fundamental class of regularization methods is \emph{convex variational regularization}, based on a convex functional $\Omega \colon X \to [0,\infty]$. Here, right inverses are approximated by $\Omega$-minimizing solutions of $A(x)=y$, namely elements in $\argmin \{ \Omega(x) \mid A(x)=y \}$. Convex variational regularization is typically realized by minimizing the Tikhonov functional $\norm{A(x) - y^\delta}^2 + \alpha \Omega(x)$ for given data $y^\delta \in Y$ and regularization parameter $\alpha > 0$.

For illustrative purpose, we restrict to norm-based Tikhonov regularization in Banach spaces
\begin{equation} \label{eq:tikhonov-banach}
B_\alpha(y^\delta) \coloneqq \argmin_{x \in X} \{ \norm{A(x) - y^\delta}^p + \alpha \norm{x}^q \} \,,
\end{equation}
where $p,q>1$.
For that, recall the Radon–Riesz (Kadec–Klee) property of a Banach space $X$: for any sequence $(x_k)_k \in X^\N$, weak convergence $x_k \rightharpoonup x$ together with norm convergence $\|x_k\| \to \|x\|$ implies strong convergence $x_k \to x$ (see, e.g., \cite{megginson1998banach,ito2014inverse}). Hilbert spaces and $L^p$ spaces with $1<p<\infty$ satisfy the Radon–Riesz property (indeed, all uniformly convex Banach spaces do). It is, however, not satisfied for $L^p$ with $p=1$ or $p=\infty$.

\begin{theorem}[Tikhonov regularization in Banach spaces] \label{thm:var-convex}
Let $X$ be reflexive, strictly convex, and satisfy the Radon–Riesz  property. Then:
\begin{enumerate}[label=(\alph*)]
\item $A^\dag \colon \ran(A) \to X \colon y  \mapsto \argmin\{\norm{x} \mid A(x) = y\}$ is well-defined.

\item  $ \forall \alpha > 0 \colon B_\alpha \colon Y \to X$ is well-defined by \eqref{eq:tikhonov-banach} and continuous.
\item If $\alpha_0(\delta) \to 0$ and $\delta^p / \alpha_0(\delta) \to 0$ as $\delta \to 0$, then $\mathcal{B}(\delta, y^\delta) = B_{\alpha_0(\delta)}(y^\delta)$ defines a regularization method for \eqref{eq:ip} on $A^\dag(X)$ with respect to $\norm{\cdot}$.
\end{enumerate}
\end{theorem}

\begin{proof}
See \cite{ivanov2020theory,scherzer2009variational}.
\end{proof}

In the Hilbert space setting, the mapping $A^\plus$ in Theorem~\ref{thm:var-convex} coincides with the Moore–Penrose inverse; see Proposition~\ref{prop:right-hil} and the subsequent discussion.

\begin{remark}[Residual convergence and rates]
In the classical Hilbert–space quadratic Tikhonov case, the source condition $x^\dag = (A^*A)^\nu w$ (with $\nu>0$) yields the familiar rate
\begin{align*}
\|x_{\alpha(\delta)}^{\delta}-x^\dag\| &= \mathcal{O}\bigl(\delta^{\tfrac{2\nu}{2\nu+1}}\bigr),\\
\|A x_{\alpha(\delta)}^{\delta}-y^\delta\| &= \mathcal{O}(\delta)
\end{align*}
for the choice $\alpha(\delta)\simeq\delta^{\tfrac{2}{2\nu+1}}$. 
Thus, in such a situation the residual converges at the order of the noise level. In Banach settings one obtains similar rates under appropriate assumptions \cite{hofmann2007convergence,resmerita2005regularization}.
\end{remark}

\section{Learned reconstruction} 
\label{sec:dc}

Now we turn our attention to learned reconstruction for \eqref{eq:ip}, where $A \colon X \to Y$ is a linear map between Banach spaces such that the null space is complemented with  $\ker(A) = \operatorname{ran}(\mathcal{P}_{\ker(A)})$.  We restrict our attention to two-step methods where the solution to \eqref{eq:ip} is defined by a reconstruction map 
\(g_{\theta, \delta} = f_\theta  \circ \mathcal{B}_\delta\), 
where $f_\theta \colon X \to X$ is an image-space network and $\mathcal{B}_\delta \colon Y \to X$ is an initial reconstruction map. We aim to derive methods that are regularization methods in the sense above. DC \eqref{eq:DC} turns out to be the key property for establishing both convergence and convergence rates. It will be realized through \emph{null-space networks} \(f_\theta\), which are allowed to act only in the null space of $A$.

\subsection{Null-space networks}
\label{sec:nullspace}

The idea of post-processing networks is to improve a given initial reconstruction map $\mathcal{B}_\delta$ by applying a neural network.
When $\mathcal{B}_\delta$ pointwise  converges to a right inverse, the initial reconstruction is DC (compare Remark \ref{rem:DC}). However, standard post-processing networks generally destroy DC. \emph{Null-space networks} form a natural class of architectures that preserve DC by construction.

\begin{definition}[Null-space network]\label{def:nsn}
For any family of Lipschitz continuous functions $\mathcal{U}_\theta \colon X \to X$ we call the family 
\begin{equation} \label{eq:null-neta} 
(f_\theta)_{\theta \in \Theta}
\qquad \text{ with } 
f_\theta \coloneqq \id_{X} + \mathcal{P}_{\ker(A)}  \mathcal{U}_\theta
\end{equation}
a \emph{null-space network} for $A$.   
\end{definition}

Sometimes by some abuse of notation we also refer to the individual mappings $f_\theta = \id_{X} + \mathcal{P}_{\ker(A)} \circ \mathcal{U}_\theta$ as null-space networks. Every such map  $f_\theta$ preserves DC in the sense that
\begin{equation} \label{eq:dc-null}
\norm{ A(x) - y} \leq \eps \quad \Rightarrow 
\quad 
\norm{ A(f_\theta(x)) - y} \leq \eps
\end{equation}
By contrast, a standard residual network $\id_{X} + \mathcal{U}_\theta$ used for post-processing does not, in general, satisfy \eqref{eq:dc-null}.

\begin{definition}[Regularizing null space network]
 Let $(f_\theta)_{\theta \in \Theta}$ be a null-space network and $(\mathcal{B}_\delta)_{\delta > 0}$ a regularization method for~\eqref{eq:ip} on the signal class $\mathcal{D}_0$. We then refer to $(f_\theta \circ \mathcal{B}_\delta)_{\delta > 0}$ as a \emph{regularizing null-space network}.
 \end{definition}

\begin{example}
$X$ and $Y$ be Hilbert spaces.  
A family $(g_\alpha)_{\alpha > 0}$ of functions $g_\alpha \colon [0, \|A^* A\|] \to \R$ is called a \emph{regularizing filter} if:
\begin{tritemize}
\item for each $\alpha > 0$, $g_\alpha$ is piecewise continuous;
\item there exists $C > 0$ such that $\sup_{\alpha > 0,\, \lambda \in [0, \|A^* A\|]} |\lambda g_\alpha(\lambda)| \le C$;
\item for all $\lambda \in (0, \|A^* A\|]$, $\lim_{\alpha \to 0} g_\alpha(\lambda) = 1 / \lambda$.
\end{tritemize}
For such a filter, define $B_\alpha \coloneqq g_\alpha(A^* A) A^*$.  
Then for a suitable parameter choice $\alpha = \alpha(\delta, y)$, the family $(B_{\alpha(\delta, \edot)})_{\delta > 0}$ forms a regularization method on $\ran(A^+)$.  
By Theorem~\ref{thm:conv}, the composed family
\[
(f_\theta \circ B_{\alpha(\delta, \edot)})_{\delta > 0}
\]
is a regularization method on $f_\theta(\ran(A^+))$.  
\end{example}

\begin{remark}[Computation of the projection layer]
A key component of a null-space network is the computation of the null-space projection $\mathcal{P}_{\ker(A)}$. In some cases, the projection can be computed explicitly. For example, for the subsampled Fourier transform $A=\mathcal{S}_I\mathcal{F}$, where $\mathcal{F}$ is the (unitary) Fourier transform and $\mathcal{S}_I$ selects frequency indices in $I$, one has $\mathcal{P}_{\ker(A)} = \mathcal{F}^*(I - \mathcal{S}_I^*\mathcal{S}_I)\mathcal{F}$ and equivalently $\mathcal{P}_{\ker(A)^\perp} = \mathcal{F}^* \mathcal{S}_I^*\mathcal{S}_I \mathcal{F}$, so $\mathcal{P}_{\ker(A)} = I - \mathcal{P}_{\ker(A)^\perp}$.

For general bounded linear forward operators between Hilbert spaces, $\mathcal{P}_{\ker(A)}$ can be computed numerically as the solution of the orthogonal projection problem $\mathcal{P}_{\ker(A)} z = \arg\min_{x \in X} \{ \|x - z\| \mid A x = 0 \}$, equivalently by minimizing $\|A x\|^2$ with initialization $x_0=z$. Landweber iterations $x_{k+1} = x_k - \tau A^*(A x_k)$ with $\tau \in (0,2/\|A\|^2)$, or conjugate gradients on the normal equations, converge to $\mathcal{P}_{\ker(A)} z$. Such strategies are particularly appealing to enforce DC  in a trained network. For end-to-end training of large-scale problems, the null-space projection layer can be expensive. In such cases, randomized SVD methods for approximating the null space  \cite{park2023fast} may provide an efficient alternative.
\end{remark}

\subsection{Convergence}
\label{sec:conv}

Throughout this section, let $(f_\theta)_{\theta \in \Theta}$ be a null-space network, let $B \colon \operatorname{ran}(A) \to X$ be a linear right inverse of $A$, and let $(\mathcal{B}_\delta)_{\delta > 0}$ be a regularization method for the inverse problem~\eqref{eq:ip} on a signal class $\mathcal{D}_0 \subseteq B(\operatorname{ran}(A))$ introduced in Definition~\ref{def:reg}.  Recall that, according to Definition \eqref{def:nsn}, the null-space network has the form $f_\theta = \mathrm{Id}_{X} + \mathcal{P}_{\ker(A)} \mathcal{U}_\theta$, with Lipschitz continuous $\mathcal{U}_\theta$.

\begin{proposition}[Learned right inverse]\label{prop:null}
The composition
\begin{equation}\label{eq:nullnet}
f_\theta \circ B \colon \ran(A) \to X \colon  
\quad y \mapsto (\id_{X} + \mathcal{P}_{\ker(A)} \circ \mathcal{U}_\theta)(B y)
\end{equation}
is a right inverse of $A$. Moreover, the following statements are equivalent:
\begin{enumerate}[label=(\roman*)]
\item $f_\theta \circ B$ is continuous,
\item $B$ is continuous,
\item $\ran(A)$ is closed.
\end{enumerate}
\end{proposition}

\begin{proof}
Since $B$ is a right inverse, $A \circ B = \id_{\ran(A)}$.  
Using the DC property \eqref{eq:dc-null}, we obtain
\[
A\bigl((\id_{X} + \mathcal{P}_{\ker(A)} \circ \mathcal{U}_\theta)(B y)\bigr) = y,
\]
which shows that $f_\theta \circ B$ is a right inverse.  
The equivalence of the continuity statements follows from standard results on linear right inverses and closed-range operators; see Proposition~\ref{prop:right-inverse-lin}.
\end{proof}

\begin{remark}[Learning the right inverse]
Right inverses defined via null-space networks can be tailored to specific signal classes, yielding an effective DC learning framework. Training typically minimizes a loss of the form
\[
\mathbb{E}\, \| (\id_{X} + \mathcal{P}_{\ker(A)} \circ \mathcal{U}_\theta)(B(y)) - x \|^2,
\]
where $B$ may represent a classical reconstruction and $\mathbb{E}$ denotes expectation with respect to a distribution on $X$.  
Irrespective of the training procedure, any $\mathcal{U}_\theta$ produces an admissible right inverse.  
The choice $\mathcal{P}_{\ker(A)} \circ \mathcal{U}_\theta = 0$ recovers the initial classical right inverse.
\end{remark}

As Proposition~\ref{prop:null} shows, solving ill-posed problems with null-space networks still requires stabilization analogous to classical regularization. The next result proves that combining a null-space network with an initial regularization method $(\mathcal{B}_\delta)_{\delta > 0}$ again yields a convergent regularization method on a signal class associated with the null-space network.

\begin{theorem}[Learned regularization]\label{thm:conv}
Consider the regularizing null-space network defined by $(f_\theta)_{\theta \in \Theta}$ and the regularization $(\mathcal{B}_\delta)_{\delta > 0}$ on the signal class $\mathcal{D}_0$.  
Then for any $\theta \in \Theta$, the family $(f_\theta \circ \mathcal{B}_\delta)_{\delta > 0}$ is a regularization method for~\eqref{eq:ip} on the signal class $f_\theta(\mathcal{D}_0)$.
\end{theorem}

\begin{proof}
Let $L$ be a Lipschitz constant of $f_\theta$.  
For any $x \in f_\theta(\mathcal{D}_0)$ and any $y^\delta \in Y$ we have $x = f_\theta(B A x)$ and
\begin{align*}
\| x - f_\theta  \circ \mathcal{B}_\delta(y^\delta) \|
&= \| f_\theta (B  A(x)) - f_\theta (\mathcal{B}_\delta(y^\delta)) \| \\
&\leq L \, \| (B  A)(x) - \mathcal{B}_\delta(y^\delta) \|.
\end{align*}
Taking the supremum over all $y^\delta$ with $\| y^\delta - y \| \le \delta$ and letting $\delta \to 0$ gives the defining property of a regularization method for~\eqref{eq:ip} on the class $f_\theta(\mathcal{D}_0)$.
\end{proof}

\subsection{Convergence rates}

In this subsection we derive quantitative error estimates for the regularizing null-space networks introduced above.  
Assume that $X$ and $Y$ are Hilbert spaces and that $B_\alpha = g_{\alpha}(A^* A) A^*$ is defined via a regularizing filter $(g_\alpha)_{\alpha > 0}$ satisfying the following conditions for some constants $\alpha_0, c_1, c_2 > 0$:
\begin{enumerate}[label=(R\arabic*)]
\item \label{it:r1}  $\forall \alpha > 0\; \forall \lambda \in [0, \|A^* A\|] \colon \lambda^\mu |1 - \lambda g_\alpha(\lambda)| \le c_1 \alpha^\mu$,
\item \label{it:r2}  $\forall \alpha > 0 \colon \| g_\alpha \|_\infty \le c_2/\alpha$.
\end{enumerate}
Let $\alpha \colon Y \times (0, \infty) \to (0, \infty)$ satisfy $\alpha^\star(\delta) \asymp (\delta/\rho)^{\frac{2}{2\mu + 1}}$ as $\delta \to 0$ (for $\mu, \rho > 0$), i.e., there exist constants $d_1,d_2>0$ such that
\[
d_1 \, (\delta/\rho)^{\frac{2}{2\mu + 1}}
\le 
\alpha^\star(\delta)
\le 
d_2 \, (\delta/\rho)^{\frac{2}{2\mu + 1}}.
\]
Define $\mathcal{B}_\delta = B_{\alpha(\delta, \edot)}$.

\begin{theorem}[Convergence rates for learned regularizations]\label{thm:rates}
Let $(f_\theta)_{\theta \in \Theta}$ be a null-space network and consider the regularizing network $(f_\theta \circ \mathcal{B}_\delta)_{\delta > 0}$ with $\mathcal{B}_\delta$ defined by a filter satisfying \ref{it:r1}–\ref{it:r2}.  
Then for $\mu,\rho>0$ and $\theta \in \Theta$ there exists $c>0$ such that
\begin{multline}\label{eq:rates}
\forall x \in \mathcal{D}_{\mu,\rho,\theta} 
\coloneqq    f_\theta(A^* A)^\mu(\overline{B_\rho(0)}) \colon \\
\sup\Bigl\{ \| \mathcal{B}_\delta(y^\delta) - x \|
\ \Big|\ 
y^\delta \in Y,\ 
\|A x - y^\delta\| \le \delta \Bigr\}
\le
c \, \delta^{\frac{2\mu}{2\mu + 1}} \rho^{\frac{1}{2\mu + 1}}.
\end{multline}
\end{theorem}

\begin{proof}
We note that
\begin{align*}
\mathcal{P}_{\ran(A^+)} \mathcal{D}_{\mu,\rho,\theta}
&= (A^* A)^\mu(\overline{B_\rho(0)})
\\
\mathcal{P}_{\ran(A^+)} \mathcal{B}_\alpha
&= g_\alpha(A^* A) A^*.
\end{align*}
Let $x \in \mathcal{D}_{\mu,\rho,\theta}$ and $y^\delta \in Y$ satisfy $\|A x - y^\delta\| \le \delta$.  
By the construction, $g_\alpha(A^* A)A^*$ defines an order-optimal regularization on $(A^*A)^\mu(\overline{B_\rho(0)})$ (see \cite{engl1996regularization}).  
Hence there exists $C>0$ independent of $x$ and $y^\delta$ such that
\[
\| g_{\alpha^\star(\delta,y^\delta)}(A^*A)A^*(y^\delta)
- \mathcal{P}_{\ran(A^+)} x \|
\le
C \, \delta^{\frac{2\mu}{2\mu + 1}} \rho^{\frac{1}{2\mu + 1}}.
\]
Using the Lipschitz continuity of $f_\theta = \id_X + \mathcal{U}_\theta$ with a constant $L$ we obtain 
\begin{align*}
&\| \mathcal{B}_{\alpha^\star(\delta,y^\delta)}(y^\delta) - x \|
\\
&=
\| (\id_X + \mathcal{U}_\theta) g_{\alpha^\star(\delta,y^\delta)}(A^*A)A^*(y^\delta)
- (\id_X + \mathcal{U}_\theta) \mathcal{P}_{\ran(A^+)} x \| \\
&\le
L \, \| g_{\alpha^\star(\delta,y^\delta)}(A^*A)A^*(y^\delta)
- \mathcal{P}_{\ran(A^+)} x \| \\
&\le
L C \, \delta^{\frac{2\mu}{2\mu + 1}} \rho^{\frac{1}{2\mu + 1}}.
\end{align*}
Taking the supremum over $x$ and $y^\delta$ with $\|A x - y^\delta\| \le \delta$ yields \eqref{eq:rates}.
\end{proof}

Filters corresponding to truncated SVD and Landweber iteration satisfy the assumptions of Theorem~\ref{thm:rates}.  
For Tikhonov regularization, these conditions hold for $\mu \le 1$.  
In particular, under the source condition
\(
x \in   \mathcal{U}_\theta(\ran(A^+)),
\)
we obtain the rate
\(
\| \mathcal{B}_{\delta}(y^\delta) - x \|
= \mathcal{O}(\delta^{1/2}).
\)

\subsection{Numerical example}
\label{sec:example}

In this subsection we present a simple example illustrating the benefits of DC networks compared with their non–DC counterparts (standard residual networks).

\paragraph*{Problem formulation:}

We work in a discrete setting with $X = Y = \mathbb{R}^{64 \times 64}$ and consider the forward problem
\begin{equation}
\label{eq:ip-num}
A(x) = M K(x),
\end{equation}
where $M$ denotes a subsampling mask consisting of vertical stripes at indices $(4k+1,4k+2)$ for $k \in \set{0,1,2,3}$, and 
\(
(Kx)_{i_1, i_2} = \sum_{j_1=1}^{i_1} x_{j_1,i_2}
\)
is the discrete integration operator in the vertical direction.  
The inverse problem \eqref{eq:ip-num} therefore has a nontrivial null space and reflects the ill-posedness of numerical integration on the complement of the null space. In particular, we have $MK(x) = K (Mx)$, and the null space of $x \mapsto Mx$ is invariant under $K$. We use two-step network $f_\theta \circ  B_\alpha$ for reconstructing $x$ from noisy data $y=A(x) + z$.    

\begin{figure}[htb!]
\includegraphics[width=\textwidth]{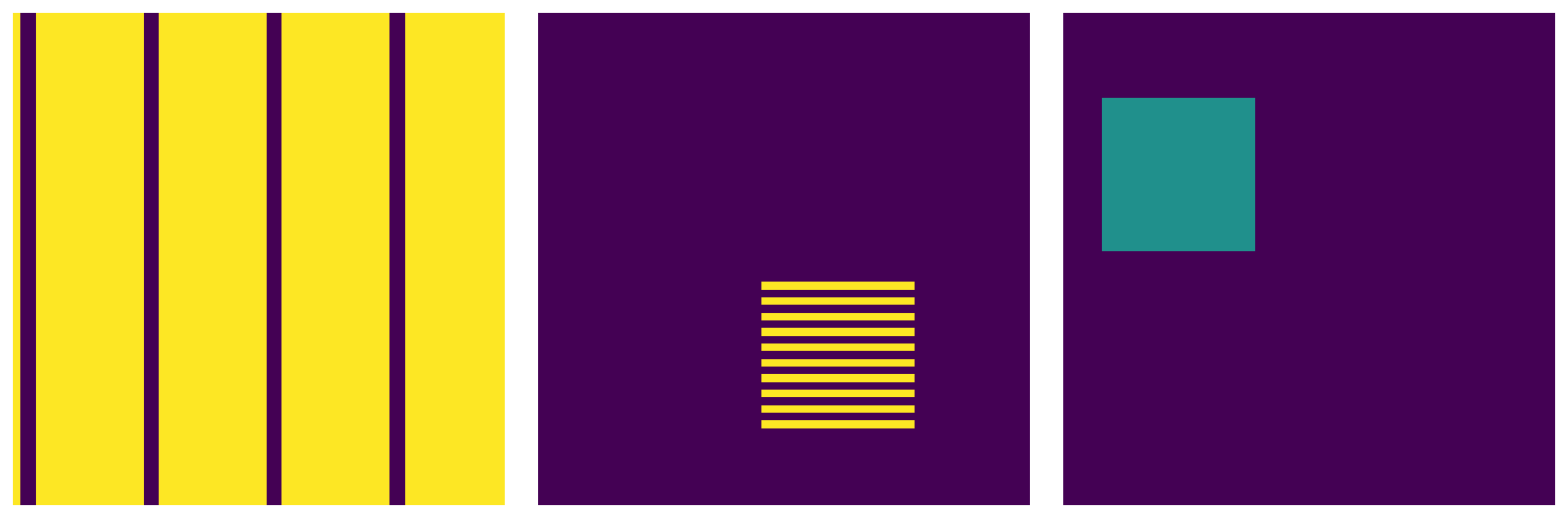}
\caption{Mask (left), ID  sample (middle), and OOD sample (right).}
\label{fig:mask_samples}
\end{figure}

All experiments are implemented in \texttt{PyTorch} using two-dimensional signals of size $64\times 64$.  
The DC of synthetic images containing a single $20\times 20$ square patch.
For ID  samples, the square is composed of horizontal stripes with alternating pixel values $(1,0,1,0,\dots)$, whereas for OOD evaluation the square has constant intensity $1/2$. The position of the square is chosen randomly, restricted to even pixel coordinates.   The observed data are given by $y = M K x + \varepsilon$, with additive Gaussian noise $\varepsilon$ of standard deviation $\sigma = 0.05$.  
Figure~\ref{fig:mask_samples} shows the mask (left), a sample from the training data (middle), and an OOD sample (right).

\paragraph*{Architecture and training:}

The two-step networks $f_\theta \circ B_\alpha$ use Tikhonov regularization $B_\alpha = (A^\top A + \alpha \id_X)^{-1} A^\top$  with $\alpha = 0.01$ as the initial reconstruction method and then apply either a standard residual network (ResNet) 
or a DC ResNet,
\begin{align}
\label{eq:num-res}
f_\theta(x) &= x + \mathcal{U}_\theta(x), \\
\label{eq:num-dc}
f_\theta(x) &= x + (I - M)\,\mathcal{U}_\theta(x),
\end{align}
where $(\mathcal{U}_\theta)_{\theta \in \Theta}$ is a learned correction operator and $(I - M)$ is the projection onto the null space of~$A$.  
For $(\mathcal{U}_\theta)_{\theta \in \Theta}$ we use a standard convolutional neural network (CNN) with $L = 5$ convolutional layers and channel width $C=6$, with ReLU activations between layers.  
The first layer maps $1 \to C$ channels, the last maps $C \to 1$, and all intermediate layers map $C \to C$.  
All convolutions use a $3 \times 3$ kernel with circular padding to avoid boundary artifacts. The overall ResNet architecture is shown in Figure~\ref{fig:smallresnet}.

\begin{figure}[htb!]
\centering
\begin{tikzpicture}[font=\small]
  \definecolor{ioColor}{RGB}{80,80,80}
  \definecolor{convColor}{RGB}{33,150,243}
  \definecolor{actColor}{RGB}{76,175,80}
  \definecolor{sumColor}{RGB}{255,111,0}
  \definecolor{arrowColor}{RGB}{55,71,79}
  \definecolor{textColor}{RGB}{30,30,30}
  \definecolor{groupColor}{RGB}{120,144,156}

  \tikzset{
    every node/.style={text=textColor},
    box/.style={draw, rounded corners=2pt, align=center, minimum width=40mm, minimum height=8mm, line width=0.8pt},
    io/.style={box, fill=ioColor!10, draw=ioColor},
    conv/.style={box, fill=convColor!12, draw=convColor},
    act/.style={box, fill=actColor!12, draw=actColor},
    sum/.style={circle, draw=sumColor, fill=sumColor!10, minimum size=6mm, line width=0.8pt},
  }

  \node[io]   (in)     at (0,  0.00) {Input $x$\\1 channel};
  \node[conv] (c1)     at (0, -1.30) {Conv $3\times3$\\$1 \to 6$\\circular pad};
  \node[act]  (r1)     at (0, -2.55) {ReLU};

  \node[conv] (cmid)   at (0, -3.85) {Conv $3\times3$\\$6 \to 6$\\circular pad};
  \node[act]  (rmid)   at (0, -5.05) {ReLU};
  \node        (dots)  at (0, -5.90) {$\vdots$};

  \draw[dashed, rounded corners=3pt, draw=groupColor, line width=0.8pt]
    (-2.6, -3.25) rectangle (2.6, -6.30);

  \draw[decorate, decoration={brace, mirror, amplitude=5pt}, draw=groupColor]
    (3.25, -3.25) -- node[right, text=groupColor, align=left, xshift=4pt]
    {repeat $(5 - 2)$ times\\each: Conv $3\times3$ ($6 \to 6$)\\circular pad + ReLU}
    (3.25, -6.30);

  \node[conv] (clast)  at (0, -7.10) {Conv $3\times3$\\$6 \to 1$\\circular pad};
  \node[sum]  (sum)    at (0, -8.25) {$+$};
  \node[io]   (out)    at (0, -9.35) {Output\\$x + \text{out}$};

  \draw[->, line width=0.9pt, draw=arrowColor] (in)   -- (c1);
  \draw[->, line width=0.9pt, draw=arrowColor] (c1)   -- (r1);
  \draw[->, line width=0.9pt, draw=arrowColor] (r1)   -- (cmid);
  \draw[->, line width=0.9pt, draw=arrowColor] (cmid) -- (rmid);
  \draw[->, line width=0.9pt, draw=arrowColor] (rmid) -- (dots);
  \draw[->, line width=0.9pt, draw=arrowColor] (dots) -- (clast);
  \draw[->, line width=0.9pt, draw=arrowColor] (clast) -- (sum);
  \draw[->, line width=0.9pt, draw=arrowColor] (sum)  -- (out);

  \draw[->, line width=1pt, draw=sumColor]
    (in.west) -- ++(-2.0,0) |- (sum.west);

\end{tikzpicture}
\caption{ResNet schematic used for the numerical simulation.}
\label{fig:smallresnet}
\end{figure}
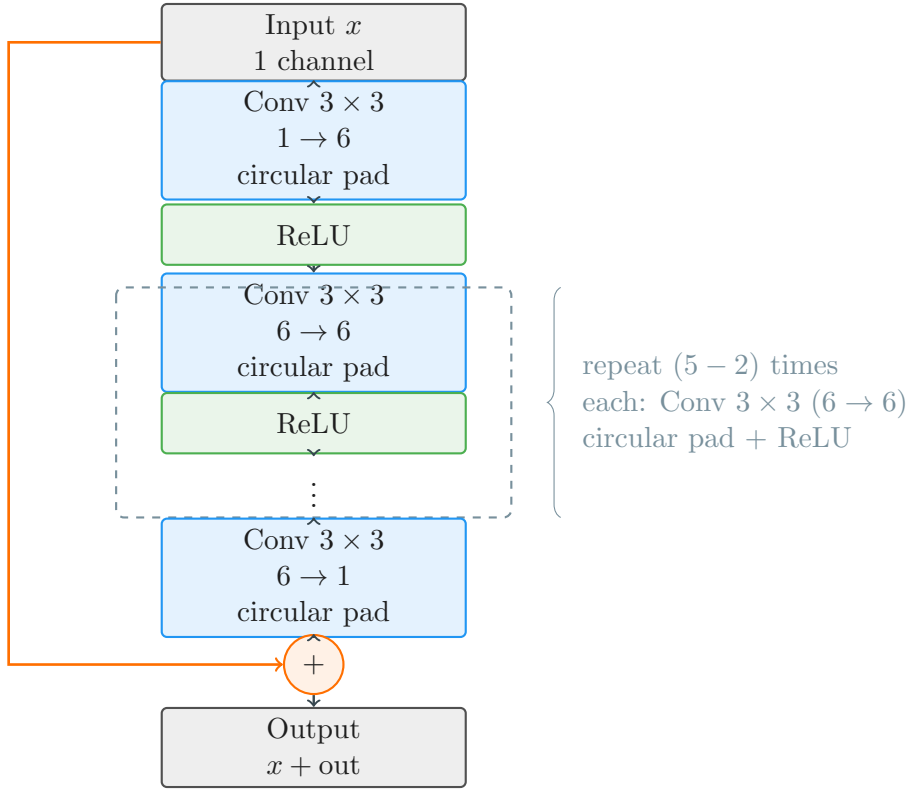

Both networks \eqref{eq:num-res} and \eqref{eq:num-dc} implement local image-to-image mappings.
They are trained for $100$ epochs using the Adam optimizer with learning rate $10^{-3}$.  
At each epoch a fresh pair $(x,y)$ is generated, where $x$ is an ID image and $y = Ax + z$ is the corresponding noisy measurement with Gaussian noise~$z$. While generating pairs on-the-fly during training is somewhat unusual, it is equivalent to pre-generating the 100 training pairs. In our setting, reuse is unnecessary due to the small number of epochs.
The training loss combines a data-fidelity term with a Frobenius-norm penalty on the network parameters,
\(
\mathcal{L}(\theta)
= \| f_\theta(B_\alpha y) - x \|_2^2
+ \lambda \|\theta\|_F^2
\),
where $\|\theta\|_F^2$ denotes the sum of squared weights across all convolutional layers and $\lambda = 10^{-4}$.  
The Frobenius penalty is implemented via the \texttt{weight\_decay} parameter of Adam to ensure uniform regularization across all convolutional weights.

\begin{figure}[htb!]
\includegraphics[width=\textwidth]{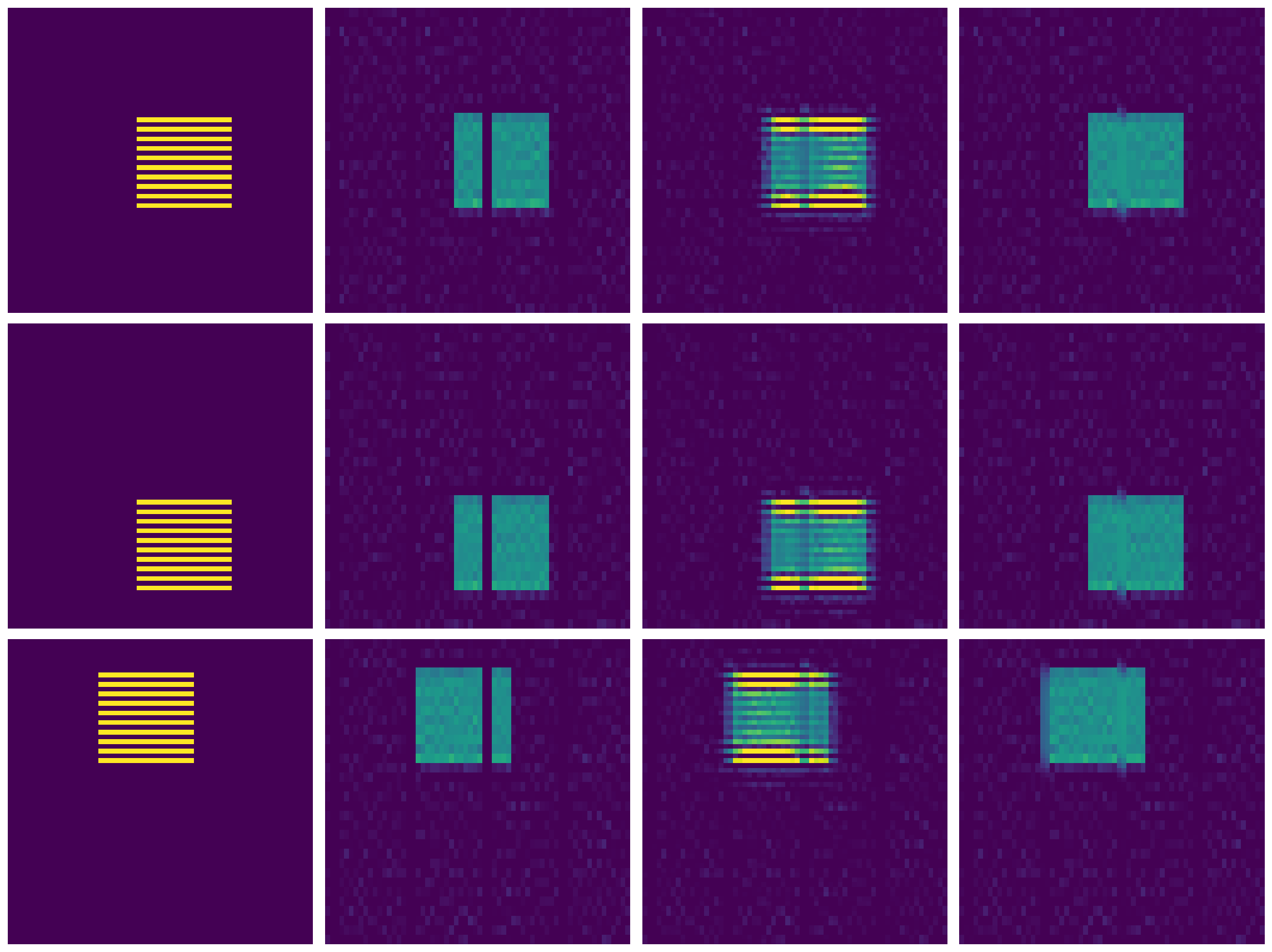}
\caption{ID evaluation for three random samples. From left to right: ground truth, Tikhonov regularizaton, residual network or a DC (null-space) network.}
\label{fig:result_ID}
\end{figure}

\begin{figure}[htb!]
\includegraphics[width=\textwidth]{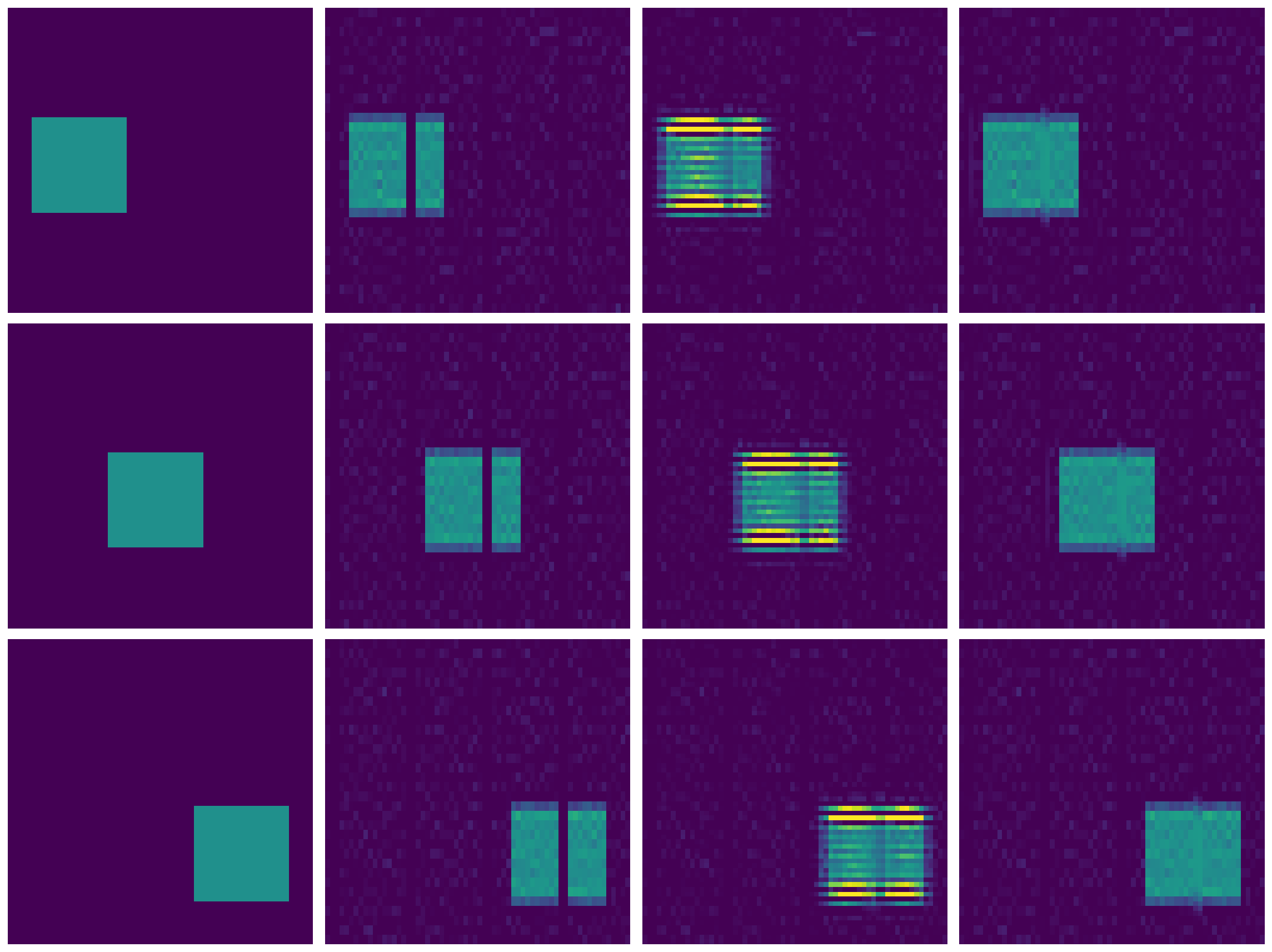}
\caption{OOD evaluation for three random samples. From left to right: ground truth, Tikhonov regularizaton, residual network or a DC (null-space) network.}
\label{fig:result_OOD}
\end{figure}

\paragraph*{Numerical results:}

The trained models are evaluated on both in-distribution (ID) striped squares and out-of-distribution (OOD) constant squares (Figure~\ref{fig:mask_samples}). Reconstructions are compared across Tikhonov, the standard ResNet, and the DC ResNet. Figures~\ref{fig:result_ID} and~\ref{fig:result_OOD} show representative reconstructions for both settings. Table~\ref{tab:combined_metrics_id_ood} reports peak signal-to-noise ratio (PSNR) and structural similarity (SSIM) for both ID and OOD samples. Higher PSNR/SSIM indicate better reconstruction quality, while lower MSE indicates better fidelity. With data range $L=1$, PSNR is computed as $\mathrm{PSNR}(x,y)=10\log_{10}\!\big(L^2/\mathrm{MSE}(x,y)\big)$.

\begin{table}[htbp]
\centering
\caption{Average reconstruction metrics over 20 samples (data range $= 1.0$). Best values in bold.}
\label{tab:combined_metrics_id_ood}
\setlength{\tabcolsep}{6pt}
\renewcommand{\arraystretch}{1.1}
\small
\begin{tabular}{l rr rr}
\toprule
Method  &
\makecell{PSNR\\(ID)} & \makecell{SSIM\\(ID)} &
\makecell{PSNR\\(OOD)} & \makecell{SSIM\\(OOD)} \\
\midrule 
Tikhonov & 15.517 & 0.3798 & 23.726 & 0.4341 \\
ResNet   & \textbf{19.218} & \textbf{0.5294} & 19.764 & \textbf{0.4802} \\
DC-Net   & 15.926 & 0.3710 & \textbf{27.336} & 0.4444 \\
\bottomrule
\end{tabular}
\end{table}

On ID data, the standard ResNet achieves the lowest error and highest perceptual quality, indicating strong reconstruction within the training distribution. The DC network improves over Tikhonov but does not match the plain ResNet on ID. On OOD data, the DC network attains the highest PSNR, reflecting robust measurement consistency and reduced hallucination under distribution shift; SSIM differences across methods are moderate. 

These results underline both the strengths and limitations of null-space networks. This motivates variants that also modify the range component~\cite{goppel2023data,goppel2025data}. While these approaches are theoretically attractive, concrete numerical advantages still need to be demonstrated. In any case, enforcing DC remains a central requirement.

\section{Summary and Outlook}
\label{sec:DC-conclusion}

Inverse problems are inherently ill-posed. Classical regularization methods address this by selecting, among all consistent candidates, the most regular solution that fits the data up to the noise level. In contrast, standard deep learning architectures such as U-Nets, ResNets, or encoder–decoder CNNs do not inherently enforce DC, and their outputs may violate the measurement equation. While such networks often produce visually superior results, they risk introducing artifacts that are inconsistent with the actual observations. To address this, recent research has focused on incorporating DC directly into network architectures. The key insight is that starting from a well-understood regularization method with proven convergence guarantees and augmenting it with a DC learning component preserves the theoretical stability of the classical approach while leveraging the expressive power of neural networks. Learned DC methods are as reliable as classical regularization, while offering superior reconstruction quality in practice.

Regularizing null-space networks are of the form $f_\theta \circ \mathcal{B}_\delta$, where $\mathcal{B}_\delta$ is a classical regularization operator and the network $f_\theta$ acts only in the null space of $A$. This structure guarantees exact preservation of the DC of $\mathcal{B}_\delta$. In practice, however, this may be too restrictive, especially in the presence of noise or model mismatch. Empirically, it is often advantageous to allow the network to make mild corrections also on the complement of the null space, thereby balancing reconstruction fidelity and stability. This leads naturally to data-proximal networks~\cite{goppel2023data,goppel2025data}, which generalize null-space networks by relaxing the strict DC constraint. Regularizing families of networks were introduced in~\cite{schwab2020big}, where it was shown that such constructions define convergent and stable regularization methods, admitting convergence rates analogous to classical schemes. The framework allows the integration of data-driven corrections while retaining theoretical guarantees.

In the finite-dimensional setting, several related concepts have been proposed: Deep decomposition learning~\cite{chen2019deep} generalizes the null-space principle by decomposing reconstructions into range and null-space components. Projectional networks~\cite{dittmer2019projectional} introduce iterative projection steps onto approximate DC sets. Soft or relaxed DC networks~\cite{huang2020data,kofler2019neural} replace strict DC with a data-fidelity term in the loss, yielding models that remain close to the measured data within a confidence region. In each variant, the strict null-space constraint is relaxed to permit limited modifications in the complement of the null space, improving robustness to noise and modeling errors.

Finally, the null-space idea has been extended to non-linear forward operators $A$, where the projection $\mathcal{P}_{\ker(A)}$ is replaced by local or learned approximations of the tangent null space. Examples include DC neural reconstruction schemes for non-linear tomography and MRI~\cite{boink2020data}. Furthermore, hybrid methods combining null-space learning with analytic reconstruction tools, such as shearlets or learned variational regularizers, have demonstrated state-of-the-art performance in ill-posed imaging problems~\cite{bubba2019learning}.

\section*{Acknowledgment}
This work was supported by the National Research Foundation of Korea(NRF) grant funded by the Korea government(MSIT) (RS-2024-00333393).


\begin{thebibliography}{10}

\bibitem{acar1994analysis}
R.~Acar and C.~R. Vogel.
\newblock Analysis of bounded variation penalty methods for ill-posed problems.
\newblock {\em Inverse Probl.}, 10(6):1217--1229, 1994.

\bibitem{adler2017solving}
J.~Adler and O.~{\"O}ktem.
\newblock Solving ill-posed inverse problems using iterative deep neural
  networks.
\newblock {\em Inverse Probl.}, 33(12):124007, 2017.

\bibitem{arridge2019solving}
S.~Arridge, P.~Maass, O.~{\"O}ktem, and C.~Sch{\"o}nlieb.
\newblock Solving inverse problems using data-driven models.
\newblock {\em Acta Numer.}, 28:1--174, 2019.

\bibitem{bhadra2021hallucinations}
S.~Bhadra, V.~A. Kelkar, F.~J. Brooks, and M.~A. Anastasio.
\newblock On hallucinations in tomographic image reconstruction.
\newblock {\em IEEE Trans. Med. Imaging}, 40(11):3249--3260, 2021.

\bibitem{boink2020data}
Y.~E. Boink, M.~Haltmeier, S.~Holman, and J.~Schwab.
\newblock Data-consistent neural networks for solving nonlinear inverse
  problems.
\newblock {\em Inverse Probl. Imaging}, 17(1):203--229, 2023.

\bibitem{bubba2019learning}
T.~A. Bubba, G.~Kutyniok, M.~Lassas, M.~Maerz, W.~Samek, S.~Siltanen, and
  V.~Srinivasan.
\newblock Learning the invisible: A hybrid deep learning-shearlet framework for
  limited angle computed tomography.
\newblock {\em Inverse Probl.}, 35(6):064002, 2019.

\bibitem{chen2019deep}
D.~Chen and M.~E. Davies.
\newblock Deep decomposition learning for inverse imaging problems.
\newblock In {\em European Conference on Computer Vision}, pages 510--526.
  Springer, 2020.

\bibitem{daubechies2004iterative}
I.~Daubechies, M.~Defrise, and C.~De~Mol.
\newblock An iterative thresholding algorithm for linear inverse problems with
  a sparsity constraint.
\newblock {\em Comm. Pure Appl. Math.}, 57(11):1413--1457, 2004.

\bibitem{dittmer2019projectional}
S.~Dittmer and P.~Maass.
\newblock A projectional ansatz to reconstruction.
\newblock {\em arXiv:1907.04675}, 2019.

\bibitem{engl1996regularization}
H.~W. Engl, M.~Hanke, and A.~Neubauer.
\newblock {\em Regularization of inverse problems}, volume 375.
\newblock Kluwer Academic Publishers Group, Dordrecht, 1996.

\bibitem{gilton2021deep}
D.~Gilton, G.~Ongie, and R.~Willett.
\newblock Deep equilibrium architectures for inverse problems in imaging.
\newblock {\em IEEE Trans. Comput. Imaging}, 7:1123--1133, 2021.

\bibitem{goppel2023data}
S.~G{\"o}ppel, J.~Frikel, and M.~Haltmeier.
\newblock Data-proximal null-space networks for inverse problems.
\newblock {\em arXiv:2309.06573}, 2023.

\bibitem{goppel2025data}
S.~G{\"o}ppel, J.~Frikel, and M.~Haltmeier.
\newblock Data-proximal neural networks for limited-view ct.
\newblock In {\em BVM Workshop}, pages 185--190. Springer, 2025.

\bibitem{grasmair2008sparse}
M.~Grasmair, M.~Haltmeier, and O.~Scherzer.
\newblock Sparse regularization with {$l^q$} penalty term.
\newblock {\em Inverse Probl.}, 24(5):055020, 13, 2008.

\bibitem{han2016deep}
Y.~Han, J.~J. Yoo, and J.~C. Ye.
\newblock Deep residual learning for compressed sensing {CT} reconstruction via
  persistent homology analysis, 2016.
\newblock http://arxiv.org/abs/1611.06391.

\bibitem{hertrich2025learning}
J.~Hertrich, H.~S. Wong, A.~Denker, et~al.
\newblock Learning regularization functionals for inverse problems: A
  comparative study.
\newblock {\em arXiv:2510.01755}, 2025.

\bibitem{hofmann2007convergence}
B.~Hofmann, B.~Kaltenbacher, C.~Poeschl, and O.~Scherzer.
\newblock A convergence rates result for tikhonov regularization in banach
  spaces with non-smooth operators.
\newblock {\em Inverse Probl.}, 23(3):987, 2007.

\bibitem{huang2020data}
Y.~Huang, A.~Preuhs, M.~Manhart, G.~Lauritsch, and A.~Maier.
\newblock Data consistent ct reconstruction from insufficient data with learned
  prior images.
\newblock {\em arXiv:2005.10034}, 2020.

\bibitem{ito2014inverse}
K.~Ito and B.~Jin.
\newblock {\em Inverse Problems: Tikhonov Theory and Algorithms}, volume~22.
\newblock World Scientific, Singapore, 2014.

\bibitem{ivanov2020theory}
V.~K. Ivanov, V.~V. Vasin, and V.~P. Tanana.
\newblock {\em Theory of linear ill-posed problems and its applications}.
\newblock Inverse and Ill-posed Problems Series. VSP, Utrecht, second edition,
  2002.
\newblock Translated and revised from the 1978 Russian original.

\bibitem{jin2017deep}
K.~H. Jin, M.~T. McCann, E.~Froustey, and M.~Unser.
\newblock Deep convolutional neural network for inverse problems in imaging.
\newblock {\em IEEE Trans. Image Process.}, 26(9):4509--4522, 2017.

\bibitem{kirisits2025regularization}
C.~Kirisits, B.~Mejri, S.~Pereverzev, O.~Scherzer, and C.~Shi.
\newblock Regularization of nonlinear inverse problems--from functional
  analysis to data-driven approaches.
\newblock {\em arXiv:2506.17465}, 2025.

\bibitem{kofler2019neural}
A.~Kofler, M.~Haltmeier, T.~Schaeffter, M.~Kachelrie{\ss}, M.~Dewey, C.~Wald,
  and C.~Kolbitsch.
\newblock Neural networks-based regularization of large-scale inverse problems
  in medical imaging.
\newblock {\em Phys. Med. Biol.}, 66(24):245020, 2021.

\bibitem{li2018nett}
H.~Li, J.~Schwab, S.~Antholzer, and M.~Haltmeier.
\newblock {NETT}: Solving inverse problems with deep neural networks.
\newblock {\em Inverse Probl.}, 36(6):065005, 2020.

\bibitem{lindenstrauss1971complemented}
J.~Lindenstrauss and L.~Tzafriri.
\newblock On the complemented subspaces problem.
\newblock {\em Isr. J. Math.}, 9(2):263--269, 1971.

\bibitem{lucas2018using}
A.~Lucas, M.~Iliadis, R.~Molina, and A.~K. Katsaggelos.
\newblock Using deep neural networks for inverse problems in imaging: beyond
  analytical methods.
\newblock {\em IEEE Signal Process. Mag.}, 35(1):20--36, 2018.

\bibitem{mccann2017convolutional}
M.~T. McCann, K.~H. Jin, and M.~Unser.
\newblock Convolutional neural networks for inverse problems in imaging: A
  review.
\newblock {\em IEEE Signal Process. Mag.}, 34(6):85--95, 2017.

\bibitem{megginson1998banach}
R.~E. Megginson.
\newblock {\em An Introduction to Banach Space Theory}.
\newblock Springer, New York, 1998.

\bibitem{morozov2012methods}
V.~A. Morozov.
\newblock {\em Methods for Solving Incorrectly Posed Problems}.
\newblock Springer Science \& Business Media, Berlin, Heidelberg, 2012.

\bibitem{nashed1987inner}
M.~Z. Nashed.
\newblock Inner, outer, and generalized inverses in banach and hilbert spaces.
\newblock {\em Numer. func. anal. opt.}, 9(3-4):261--325, 1987.

\bibitem{natterer2001mathematics}
F.~Natterer.
\newblock {\em The Mathematics of Computerized Tomography}.
\newblock SIAM, Philadelphia, PA, 2001.

\bibitem{park2023fast}
T.~Park and Y.~Nakatsukasa.
\newblock A fast randomized algorithm for computing an approximate null space.
\newblock {\em BIT Numer. Math.}, 63(2):36, 2023.

\bibitem{resmerita2005regularization}
E.~Resmerita.
\newblock Regularization of ill-posed problems in banach spaces: convergence
  rates.
\newblock {\em Inverse Probl.}, 21(4):1303, 2005.

\bibitem{scherzer2009variational}
O.~Scherzer, M.~Grasmair, H.~Grossauer, M.~Haltmeier, and F.~Lenzen.
\newblock {\em Variational methods in imaging}, volume 167 of {\em Applied
  Mathematical Sciences}.
\newblock Springer, New York, 2009.

\bibitem{schuster2012regularization}
T.~Schuster, B.~Kaltenbacher, B.~Hofmann, and K.~S. Kazimierski.
\newblock {\em Regularization methods in Banach spaces}, volume~10.
\newblock Walter de Gruyter, 2012.

\bibitem{schwab2019deep}
J.~Schwab, S.~Antholzer, and M.~Haltmeier.
\newblock Deep null space learning for inverse problems: convergence analysis
  and rates.
\newblock {\em Inverse Probl.}, 35(2):025008, 2019.

\bibitem{schwab2020big}
J.~Schwab, S.~Antholzer, and M.~Haltmeier.
\newblock Big in {Japan}: Regularizing networks for solving inverse problems.
\newblock {\em J. Math. Imaging Vis.}, 62:445--455, 2020.

\bibitem{tikhonov1977solutions}
A.~N. Tikhonov and V.~Y. Arsenin.
\newblock {\em Solutions of Ill-Posed Problems}.
\newblock Winston; John Wiley \& Sons, Washington, DC; New York, 1977.
\newblock English translation of the 1974 Russian edition.

\bibitem{wang2020deep}
G.~Wang, J.~C. Ye, and B.~De~Man.
\newblock Deep learning for tomographic image reconstruction.
\newblock {\em Nat. Mach. Intell.}, 2(12):737--748, 2020.

\end{thebibliography}

\end{document}